\documentclass[11pt]{amsart}

\title[RLWE and PLWE over cyclotomic fields are not equivalent]{RLWE and PLWE over cyclotomic fields\\are not equivalent}

\author[A.~J.~Di~Scala]{Antonio J.~Di~Scala}
\address{\parbox{\linewidth}{
Politecnico di Torino, Department of Mathematical Sciences\\
Corso Duca degli Abruzzi 24, 10129 Torino, Italy\\[-8pt]}}
\email{antonio.discala@polito.it}

\author[C.~Sanna]{Carlo Sanna}
\address{\parbox{\linewidth}{
Politecnico di Torino, Department of Mathematical Sciences\\
Corso Duca degli Abruzzi 24, 10129 Torino, Italy\\[-8pt]}}
\email{carlo.sanna.dev@gmail.com}

\author[E.~Signorini]{Edoardo Signorini}
\address{\parbox{\linewidth}{
Telsy S.p.A.\\
Corso Svizzera 185, 10149 Torino, Italy\\[-8pt]}}
\email{edoardo.signorini@telsy.it}



\usepackage{amsmath}
\usepackage{amssymb}
\usepackage{amsthm}
\usepackage{booktabs}
\usepackage{color}
\usepackage{enumerate}
\usepackage[left=1.15in, right=1.15in, top=.72in, bottom=.72in]{geometry}
\usepackage{graphicx}
\usepackage[colorlinks=true]{hyperref}

\newtheorem{theorem}{Theorem}[section]
\newtheorem{corollary}{Corollary}[section]
\newtheorem{lemma}[theorem]{Lemma}

\theoremstyle{remark}

\newcommand{\Cond}{\operatorname{Cond}}

\newcommand{\Tr}{\operatorname{Tr}}
\newcommand{\Id}{\operatorname{Id}}

\begin{document}

\maketitle

\begin{abstract}
    We prove that the Ring Learning With Errors (RLWE) and the Polynomial Learning With Errors (PLWE) problems over the cyclotomic field $\mathbb{Q}(\zeta_n)$ are not equivalent.
    Precisely, we show that reducing one problem to the other increases the noise by a factor that is more than polynomial in $n$.
    We do so by providing a lower bound, holding for infinitely many positive integers $n$, for the condition number of the Vandermonde matrix of the $n$th cyclotomic polynomial.
\end{abstract}

\section{Introduction}

Since the theoretical results of Ajtai~\cite{Ajt96}, lattice-based cryptography has gained increasing interest. 
Indeed, numerous lattice-based encryption and digital signature schemes, with performance comparable or even superior to that of their number-theoretic counterparts, have been proposed~\cite{ADPS16, BDK+18, DKL+18, HRSS17}.
In particular, because of their presumed resistance against quantum attacks, lattice-based proposals are the most numerous in the final phase of the NIST post-quantum standardization process, with finalist candidates in both key encapsulation~\cite{ABD+20, BBD+20, CDR+20} and digital signature schemes~\cite{BDK+21,FHK+20}.

The main building block of lattice-based cryptographic schemes is the Learning With Errors (LWE) problem~\cite{Reg05}, which, roughly speaking, consists of retrieving a secret vector $s \in \mathbb{Z}_q^n$ from a noisy random sample of matrix products.
On the one hand, LWE-based encryption schemes enjoy good computational efficiency and solid theoretical security bases.
On the other hand, they require the ciphertexts or the public keys to be nearly quadratic with respect to the security parameters.
To overcome this inefficiency, algebraic variants of the LWE problem have been introduced, which consider the problem no longer over $\mathbb{Z}_q$ but over the quotient ring $\mathbb{Z}_q[X]/(f)$, where $f\in \mathbb{Z}_q[X]$ is a monic and irreducible polynomial.
The variant known as Polynomial-LWE (PLWE), was first proposed using power-of-two degree cyclotomic polynomials~\cite{SSTX09}. 
Later, Lyubashevsky, Peikert, and Regev~\cite{LPR10} introduced the Ring-LWE (RLWE) variant over the ring of integers $\mathcal{O}_K$ of a number field $K = \mathbb{Q}(\theta)$ (for surveys on RLWE, see~\cite{B-C2020a, ELOS2016}).

The main advantage of RLWE (and of later generalizations such as Module-LWE~\cite{LS15}) is the provable-security link with hard computational problems over (ideal) lattices, as for plain LWE.
Nevertheless, most of the concrete constructions of lattice-based schemes, while enjoying the security proofs of RLWE, are expressed in the simpler formalism of PLWE. The latter is in fact preferable in implementations, where the modular arithmetic between polynomials can be efficiently implemented.
For these reasons, it is interesting to study for which families of polynomials $f$ the RLWE and PLWE problems are equivalent, that is, every solution of the first problem can be turned in polynomial time into a solution of the second problem, and viceversa, incurring in a noise increase that is polynomial in the degree of $f$.

More precisely, let $K = \mathbb{Q}(\theta)$ be a monogenic number field of degree $m$, and let $f \in \mathbb{Z}[X]$ be the minimal polynomial of $\theta$, so that $\mathcal{O}_K \cong \mathbb{Z}[X]/(f)$.
The geometric notion of short element derives from a choice of a norm on $K$ by embedding the number field in $\mathbb{C}^m$.
On~the one hand, RLWE makes use of the \emph{canonical embedding} (or \emph{Minkowski embedding}) $\sigma$ from $K$ to $\mathbb{C}^m$, where $\sigma_i(\theta)$ ($i=1,\ldots,m$) are the Galois conjugates of $\theta$.
On the other hand, PLWE makes use of the \emph{coefficient embedding}, which maps each $x \in \mathcal{O}_K$ to the vector $(x_0, \ldots, x_{m-1}) \in \mathbb{Z}^m$ of its coefficients with respect to the power basis $1, \theta, \ldots, \theta^{m-1}$.
As a linear map, the canonical embedding $\sigma$ has a matrix representation $V \in \mathbb{C}^{m \times m}$, so that, for each $x \in \mathcal{O}_K$, we have $\sigma(x) = V \cdot (x_0, \dots, x_{m-1})^\intercal$.
For the equivalence between RLWE and PLWE, it is important to determine when, whether $\|x\|$ is small, then so is $\|\sigma(x)\|$, and vice versa.
This notion is quantified by $V$ having a small \emph{condition number} $\Cond(V) := \|V\| \|V^{-1}\|$, where $\|V\| := \sqrt{\Tr(V^*\!\,V)}$ is the \emph{Frobenius norm} of $V$, and $V^*$ is the conjugate transpose of $V$.
Precisely, for the equivalence of the RLWE and PLWE problems it must be $\Cond(V) = O(m^r)$ for some constant $r > 0$, depending only on the family of polynomials $f$.

An important case is that of cyclotomic fields. 
When $K = \mathbb{Q}(\zeta_n)$ is the $n$th cyclotomic field, $V_n := V$ is the Vandermonde matrix of the $n$th cyclotomic polynomial $\Phi_n(X)$, that is,
\begin{equation*}
    V_n := 
    \begin{pmatrix}
        1      & \zeta_{n,0}   & \zeta_{n,0}^2   & \cdots & \zeta_{n,0}^{m-1}   \\
        1      & \zeta_{n,1}   & \zeta_{n,1}^2   & \cdots & \zeta_{n,1}^{m-1}   \\
        1      & \zeta_{n,2}   & \zeta_{n,2}^2   & \cdots & \zeta_{n,2}^{m-1}   \\
        \vdots & \vdots        & \vdots          & \ddots & \vdots              \\
        1      & \zeta_{n,m-1} & \zeta_{n,m-1}^2 & \cdots & \zeta_{n,m-1}^{m-1} \\
    \end{pmatrix} ,
\end{equation*}
where $\zeta_{n,0}, \ldots, \zeta_{n,m-1}$ are the primitive $n$th roots of unity, and $m = \varphi(n)$ is the Euler totient function of $n$.
Note that $\Phi_n(X)$ has degree $m$.
If $n$ is a power of $2$, then it is easy to show that $V_n$ is a scaled isometry, so that $\Cond(V_n) = m$ and consequently RLWE and PLWE are equivalent.
Blanco‑Chac\'on~\cite{B-C2020b} (see also~\cite{B-C2021, BCLH2021}) proved that $\Cond(V_n) = O(n^{r_k})$, where $r_k > 0$ is a constant depending only on the number $k$ of distinct prime factors of $n$.
Therefore, RLWE and PLWE restricted to the positive integers $n$ with a bounded number of prime factors are equivalent.
Furthermore, in a previous work~\cite{DSS2021}, the authors gave an explicit formula for the condition number of $V_n$ when $n$ is a prime power or a power of $2$ times an odd prime power.

Our main result is the following.

\begin{theorem}\label{thm:main}
    There exist infinitely many positive integers $n$ such that
    \begin{equation*}
        \Cond(V_n) > \exp\!\big(n^{\log 2 / \log \log n}\big) / \sqrt{n} .
    \end{equation*}
    In~particular, for every fixed $r > 0$, we have that $\Cond(V_n) \neq O(n^r)$.
\end{theorem}

As a consequence of Theorem~\ref{thm:main} and the previous considerations, one immediately gets the following corollary.

\begin{corollary}
    RLWE and PLWE over cyclotomic fields are not equivalent. 
\end{corollary}

It might be interesting to determine the maximal order of $\Cond(V_n)$ and, in particular, if the lower bound of Theorem~\ref{thm:main} can be improved significantly.
For a plot of the values of $\Cond(V_n)$ up to $n=10000$, see Figure~\ref{fig:1}.
The library used for the calculation of $\Cond(V_n)$ is available in \cite{CondSoftware}.

\subsection*{Acknowledgements}

The authors are members of CrypTO, the group of Cryptography and Number Theory of Politecnico di Torino.
A.~J.~Di Scala and C.~Sanna are members of GNSAGA of INdAM.
A. J. Di Scala is a member of DISMA Dipartimento di Eccellenza MIUR 2018-2022.
E. Signorini is a cryptographer at Telsy~S.p.A.

\section{Proof of Theorem~\ref{thm:main}}

Throughout this section, let $n$ be a positive integer and put $m := \varphi(n)$.
We write $\Id_k$ for the $k \times k$ identity matrix, and we count rows and columns starting from $0$, so that the first row or column is the $0$th.
Furthermore, let
\begin{equation*}
    W_n := \begin{pmatrix}
        1      & \zeta_{n,0}   & \zeta_{n,0}^2   & \cdots & \zeta_{n,0}^{mn-1}   \\
        1      & \zeta_{n,1}   & \zeta_{n,1}^2   & \cdots & \zeta_{n,1}^{mn-1}   \\
        1      & \zeta_{n,2}   & \zeta_{n,2}^2   & \cdots & \zeta_{n,2}^{mn-1}   \\
        \vdots & \vdots        & \vdots          & \ddots & \vdots               \\
        1      & \zeta_{n,m-1} & \zeta_{n,m-1}^2 & \cdots & \zeta_{n,m-1}^{mn-1} \\
    \end{pmatrix}
\end{equation*}
be the $m \times mn$ matrix obtained by ``continuing'' $V_n$ to the right.

\begin{lemma}\label{lem:ortho}
    We have $W_n W_n^* = mn \Id_m$.
\end{lemma}
\begin{proof}
    The scalar product of the $i$th row of $W_n$ and the $j$th column of $W_n^*$ is equal to
    \begin{equation*}
        \sum_{k\,=\,0}^{mn - 1} \left(\zeta_{n, i} \overline{\zeta_{n, j}}\right)^k = \begin{cases} mn & \text{ if } i = j ; \\ 0 & \text{ if } i \neq j ; \end{cases}
    \end{equation*}
    where we used the formula for the sum of a geometric progression.
    The claim follows.
\end{proof}

Let $a_n(j)$ denote the coefficient of $X^j$ in the $n$th cyclotomic polynomial $\Phi_n(X)$, that is,
\begin{equation*}
    \Phi_n(X) = \sum_{j \,=\, 0}^m a_n(j) X^j .
\end{equation*}
The study of the coefficients of the cyclotomic polynomials has a very long history, which goes back at least to Gauss.
For a survey, see~\cite{SurveyCyclo}.
Let $A(n)$ be the maximum of the absolute values of $a_n(0), \dots, a_n(m - 1)$.
We need the following result of Vaughan~\cite{V1974}.

\begin{theorem}\label{thm:vaughan}
    We have $A(n) > \exp\!\left(n^{\log 2 / \log \log n} \right)$ for infinitely many positive integers $n$.
\end{theorem}

Let $C_n$ be the \emph{companion matrix} of $\Phi_n(X)$, which is the $m \times m$ matrix defined as
\begin{equation*}
    C_n := 
    \begin{pmatrix}
        0      & 0      & \cdots & 0      & -a_n(0)     \\
        1      & 0      & \cdots & 0      & -a_n(1)     \\
        0      & 1      & \cdots & 0      & -a_n(2)     \\
        \vdots & \vdots & \ddots & \vdots & \vdots      \\
        0      & 0      & \cdots & 1      & -a_n(m - 1) \\
    \end{pmatrix} ,
\end{equation*}
and let 
\begin{equation*}
    S_n := \big(\Id_m \mid C_n^m \mid C_n^{2m} \mid \cdots \mid C_n^{(n-1)m}\big)
\end{equation*}
be the $m \times mn$ matrix obtained by the juxtaposition of the first $n$ powers of $C_n^m$.

\begin{lemma}\label{lem:W_n-V_nS_n}
    We have $V_n^{-1} W_n = S_n$.
\end{lemma}
\begin{proof}
    Let $K := \mathbb{Q}(\zeta_{n})$ be the $n$th cyclotomic field.
    For each $k \in \{0, \dots, m-1\}$ we have that $1, \zeta_{n,k}, \zeta_{n,k}^2, \dots, \zeta_{n, k}^{m - 1}$ is a basis of $K$ over $\mathbb{Q}$.
    Moreover, multiplication by $\zeta_{n, k}$ is a $\mathbb{Q}$-linear map $K \to K$ whose transformation matrix respect to the aforementioned basis is equal to $C_n$.
    Therefore, if $z_0, \dots, z_{m-1} \in K$ satisfy
    \begin{equation*}
        \begin{pmatrix}
            z_0 \\ z_1 \\ \vdots \\ z_{m-1}
        \end{pmatrix}
        = V_n
        \begin{pmatrix}
            c_0 \\ c_1 \\ \vdots \\ c_{m - 1}
        \end{pmatrix}
    \end{equation*}
    for some $c_0, \dots, c_{m-1} \in \mathbb{Q}$, then it follows that
    \begin{equation*}
        \begin{pmatrix}
            \zeta_{n,0}^j z_0 \\ \zeta_{n,1}^j z_1 \\ \vdots \\ \zeta_{n,m-1}^j z_{m-1}
        \end{pmatrix}
        = V_n C_n^j
        \begin{pmatrix}
            c_0 \\ c_1 \\ \vdots \\ c_{m - 1}
        \end{pmatrix}
    \end{equation*}
    for every integer $j \geq 0$.
    Consequently, we have that
    \begin{equation}\label{equ:large-matrix}
        \begin{pmatrix}
            \zeta_{n,0}^j   & \zeta_{n,0}^{j+1}   & \cdots & \zeta_{n,0}^{j+m-1}   \\
            \zeta_{n,1}^j   & \zeta_{n,1}^{j+1}   & \cdots & \zeta_{n,1}^{j+m-1}   \\
            \vdots          & \vdots              & \ddots & \vdots                \\
            \zeta_{n,m-1}^j & \zeta_{n,m-1}^{j+1} & \cdots & \zeta_{n,m-1}^{j+m-1} \\
        \end{pmatrix}
        = V_n C_n^j \Id_m = V_n C_n^j,
    \end{equation}
    for every integer $j \geq 0$.
    Therefore, by juxtaposition of~\eqref{equ:large-matrix} for $j = 0, m, 2m, \dots, (n - 1)m$, we obtain that $W_n = V_n S_n$.
    The claim follows.
\end{proof}

\begin{lemma}\label{lem:norm-of-inverse-V_n}
    We have $\|V_n^{-1}\|^2 = \tfrac1{mn}\sum_{k = 0}^{n - 1} \|C_n^{km}\|^2$.
\end{lemma}
\begin{proof}
    From Lemma~\ref{lem:ortho} and Lemma~\ref{lem:W_n-V_nS_n}, it follows that
    \begin{equation*}
        mn \|V_n^{-1}\|^2 = mn \Tr\!\left(V_n^{-1} \big(V_n^{-1}\big)^*\right) = \Tr\!\left(V_n^{-1} W_n W_n^* \big(V_n^{-1}\big)^*\right) = \Tr(S_n S_n^*) .
    \end{equation*}
    Moreover, by the definition of $S_n$, we have that
    \begin{align*}
        \Tr(S_n S_n^*) & = \Tr\Big(\big(\Id_m \mid C_n^m \mid \cdots \mid C_n^{(n-1)m}\big) \begin{pmatrix} \Id_m \\ \cmidrule(lr){1-1} (C_n^m)^* \\ \cmidrule(lr){1-1} \raisebox{3pt}{\vdots} \\ \cmidrule(lr){1-1} \big(C_n^{(n-1)m}\big)^* \end{pmatrix} \Big) \\
                       & = \sum_{k \,=\, 0}^{n - 1} \Tr\!\big(C_n^{km} \big(C_n^{km}\big)^*\big) = \sum_{k \,=\, 0}^{n - 1} \|C_n^{km}\|^2 ,
    \end{align*}
    and the claim follows.
\end{proof}

\begin{lemma}\label{lem:companion-powers}
    Let $k$ be a positive integer and let
    \begin{equation*}
        C := \begin{pmatrix} 
            0      & 0      & \cdots & 0      & c_0     \\
            1      & 0      & \cdots & 0      & c_1     \\
            0      & 1      & \cdots & 0      & c_2     \\
            \vdots & \vdots & \ddots & \vdots & \vdots  \\
            0      & 0      & \cdots & 1      & c_{k-1}
        \end{pmatrix}
        \in \mathbb{C}^{k \times k} .
    \end{equation*}
    Then, for every integer $j \in [1, k]$, the $(k - j)$th column of $C^j$ is equal to $\begin{pmatrix}c_0 & c_1 & \cdots & c_{k-1} \end{pmatrix}^\intercal$.
\end{lemma}
\begin{proof}
    Actually, a stronger claim holds: For every integer $j \in [1, k]$, the $0$th, $1$th, \dots, $(k - j)$th columns of $C^j$ are equal to the $(j - 1)$th, $j$th, \dots, $(k - 1)$th columns of $C$, respectively.
    This follows easily by induction on $j$.
\end{proof}

We are ready to prove Theorem~\ref{thm:main}.
From Lemma~\ref{lem:norm-of-inverse-V_n} and Lemma~\ref{lem:companion-powers}, it follows that
\begin{equation*}
    \|V_n^{-1}\|^2 = \tfrac1{mn}\sum_{k = 0}^{n - 1} \|C_n^{km}\|^2 \geq \tfrac1{mn} \|C_n^m\|^2 \geq \tfrac1{mn} \sum_{j \,=\, 0}^{m-1} |a_n(j)|^2 \geq \tfrac1{mn} A(n)^2 .
\end{equation*}
In turn, this implies that
\begin{equation*}
    \Cond(V_n) = \|V_n\| \|V_n^{-1}\| = m \|V_n^{-1}\| \geq \sqrt{\tfrac{m}{n}} A(n) \geq \tfrac1{\sqrt{n}} A(n) .
\end{equation*}
As a consequence, Theorem~\ref{thm:vaughan} yields that
\begin{equation*}
    \Cond(V_n) > \exp\!\left(n^{\log 2 / \log \log n} \right) / \sqrt{n} ,
\end{equation*}
for infinitely many positive integers $n$.
Therefore, for every fixed $r > 0$, we have that
\begin{equation*}
    \limsup_{n \to +\infty} \frac{\Cond(V_n)}{n^r} = +\infty ,
\end{equation*}
so that $\Cond(V_n) \neq O(n^r)$.
The proof is complete.

\begin{figure}
    \includegraphics[scale=0.9]{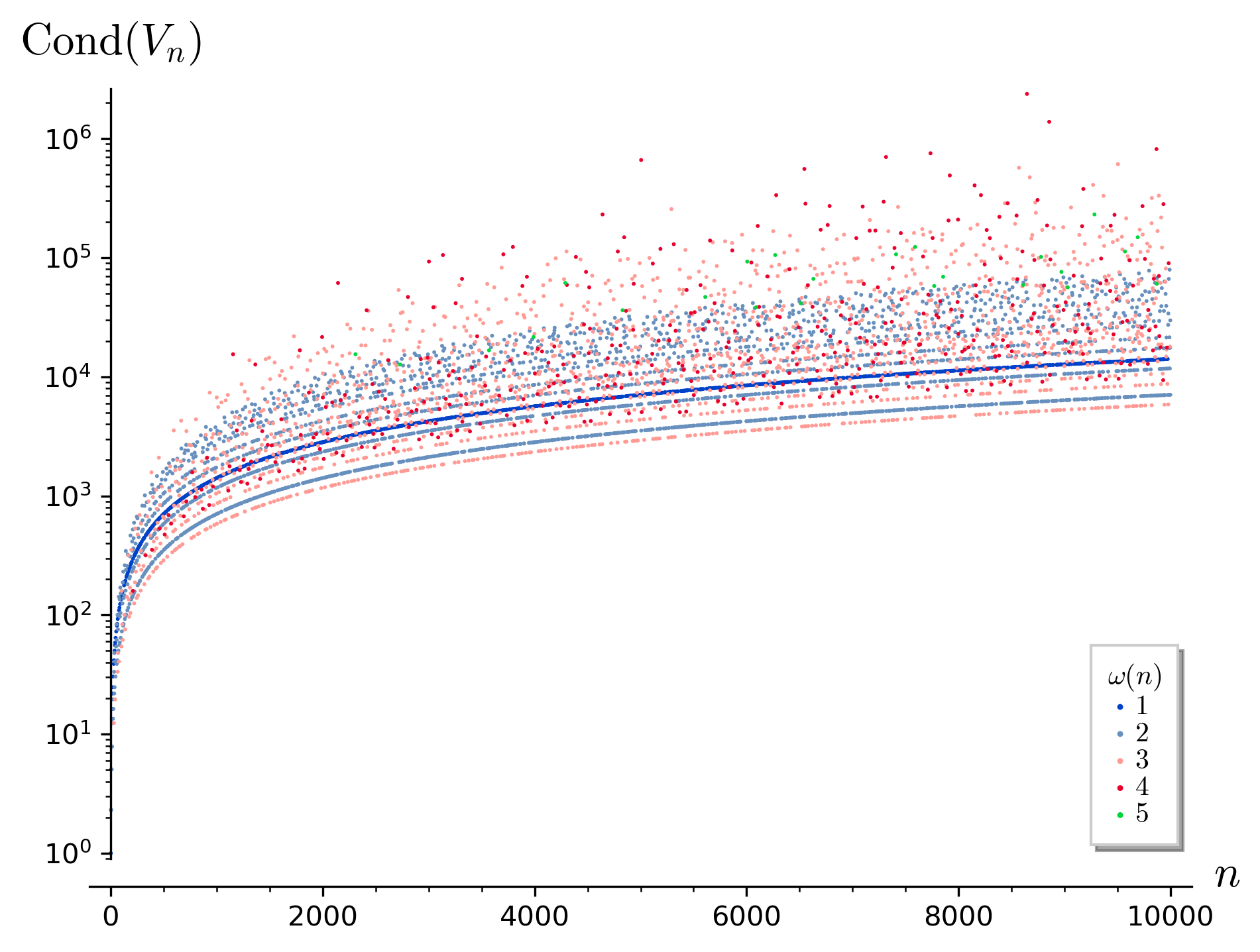}
    \caption{The condition number of $V_n$ with $n$ squarefree, $1 < n < 10000$. The data is partitioned according to the number $\omega(n)$ of prime factors of $n$.}
    \label{fig:1}
\end{figure}

\end{document}